\documentclass[12pt]{amsart}
\usepackage{amssymb}
\usepackage{paralist}
\usepackage{fancyhdr}
\usepackage{thmtools}
\usepackage{thm-restate}
\usepackage{amssymb}
\usepackage{hyperref}
\usepackage{cleveref}

\newtheorem{prop}[subsection]{Proposition}
\newtheorem{lem}[subsection]{Lemma}
\newtheorem{rem}[subsection]{Remark}
\newtheorem{coro}[subsection]{Corollary}

\newtheorem{assumption}[subsection]{Assumption}
\newtheorem{defin}[subsection]{Definition}

\newcommand{\ds}{\displaystyle}
\newcommand{\eps}{\varepsilon}
\newcommand{\Nat}{{\bf N}}
\newcommand{\Prob}{{\bf P}}
\newcommand{\E}{{\bf E}}
\newcommand{\V}{{\bf Var}}
\newcommand{\GPA}{{\textrm{GPA}_f}}
\newcommand{\MPA}{{\textrm{MPA}_f}}

\newcommand{\bpm}{\begin{pmatrix}}
\newcommand{\epm}{\end{pmatrix}}

\begin{document}

\title{Preferential Attachment Processes Approaching The Rado Multigraph}
\author{Richard Elwes}

\begin{abstract}
We consider a preferential attachment process in which a multigraph is built one node at a time. The number of edges added at stage $t$, emanating from the new node, is given by some prescribed function $f(t)$, generalising a model considered by Kleinberg and Kleinberg in 2005 where $f$ was presumed constant. We show that if $f(t)$ is asymptotically bounded above and below by linear functions in $t$, then with probability $1$ the infinite limit of the process will be isomorphic to the \emph{Rado multigraph}. This structure is the natural multigraph analogue of the Rado graph, which we introduce here.
\end{abstract}

\maketitle

\section{Introduction}

In recent decades, there has been much interest in modelling and analysing the many networks which appear in the real world, in contexts such as the world wide web or online social networks. This work has drawn heavily on the mathematical study of random graphs, a subject with its origins in the 1959 work of Erd\H{o}s and R\'enyi, \cite{ER}. They principally studied the graphs which emerge from the following process: begin with a collection of nodes, and independently connect every pair with an edge, with some fixed probability $p$.

Erd\H{o}s-R\'enyi random graph theory has two distinct facets. First, researchers have analysed the finite graphs which arise. Here, questions of interest include the emergence of a giant component and the degree distribution of the nodes, and analyses are typically highly sensitive to the value of $p$. In \cite{Bol}, Bollob\'as provides a comprehensive discussion of such matters.

The second angle of approach is to consider this process on a countably infinite set of nodes. In this case, a remarkable theorem of Erd\H{o}s and R\'enyi guarantees that, irrespective of the value of $p \in (0,1)$, the resulting graph will with probability $1$ be isomorphic to the \emph{Rado graph}. This famous graph is axiomatised by the following schema: given any finite disjoint sets of nodes $U$ and $V$, there exists a node $v$ connected to each node in $V$ and none in $U$. This graph exhibits many properties which logicians and combinatorists enjoy. To start with, it is \emph{universal} in that it isomorphically embeds every finite and countably infinite graph. It is also \emph{countably categorical}, meaning that any two countable models of the above axioms will be isomorphic. The graph is \emph{1-transitive} in that for any any two nodes $v_1, v_2$ there is an automorphism $\alpha$ where $\alpha(v_1)=v_2$. It is  \emph{ultrahomogeneous}: any isomorphism between finite induced subgraphs extends to an automorphism of the whole structure. (Analogues of these facts are proved for a new structure, the Rado multigraph, in Proposition \ref{prop:multiprop} below.) The Rado graph continues to attract the attention of today's permutation group-theorists; it is known that its automorphism group is simple (in the group-theoretic sense), and satisfies the \emph{strong small-index property}. In \cite{C}, Cameron provides a recent survey of such matters.  Beyond this, the Rado graph satisfies several subtler properties, notably \emph{rank-1 supersimplicity} and \emph{1-basedness}, which make it a central object of study for today's model theorists. Wagner provides an authoritative account in \cite{Wag}.

In more recent years, however, network science has grown beyond the Erd\H{o}s-R\'enyi approach, to embrace alternative methods for modelling real-world networks. The most prominent of these is perhaps the \emph{preferential attachment} (PA) mechanism introduced by Barab\'asi and Albert in \cite{BA}. Another notable class of models derive from the \emph{web-copying} mechanism introduced by Adler and Mitzenmacher in \cite{AM}.

In PA models, a new node is introduced at each time step, and then connected to each pre-existing node with a probability depending on the current degree of that node, according to a \emph{rich-get-richer} paradigm. PA processes can exhibit several properties observed in real-world networks (but absent in Erd\H{o}s-R\'enyi graphs), notably \emph{scale-freeness} meaning that the proportion of nodes of degree $k$ is asymptotically proportional to $k^{-\gamma}$ for some fixed $\gamma$ and all $k$.

What can we say of the infinite limits of these processes? Results of Bonato and Janssen \cite{BJ} have made significant progress for web-copying models. Less work has been done in the case of PA processes. The work of Oliveira and Spencer \cite{OS} studying the \emph{Growing Network} model of  Krapivsky and Redner \cite{KR} and of Drinea, Enachescu, and Mitzenmacher \cite{DEM} is a notable exception. Of greatest relevance to the current paper, however, is the work of Kleinberg and Kleinberg \cite{KK}. There the following PA process is considered: at each time-step, a single node and a constant number $C$ of edges are added. The new edges all emanate from the new node, with their end-points independently chosen among the pre-existing nodes, with probabilities proportional to their degrees. The resulting structures are analysed as \emph{directed multigraphs}: all edges are directed, two or more may share the same start and end-points.

Kleinberg and Kleinberg prove that if $C=1$ or $C=2$, then there exists an infinite structure, to which, with probability 1, the infinite limit of the process will be isomorphic. However, the analogous result fails for $C \geq 3$: given two instantiations of the process, there is a positive probability that their infinite limits will fail to be isomorphic.

In this paper we extend the results and methods of \cite{KK}, by considering a process which adds $f(t)$ many edges at stage $t$ for some function $f:\textbf{N} \to \textbf{N}$. Again the start-point of every edge is the new node, and the end-points are chosen independently with probability proportional to the nodes' degrees. It follows easily from the results of \cite{KK} that whenever $f$ is non-constant, or constant with value $\geq 3$, there is a positive probability that the infinite limiting structures of two instantiations will be non-isomorphic as directed multigraphs. However, by forgetting the directions of edges, and looking for isomorphisms as multigraphs, we are able to recover a new categoricity result. In Theorem \ref{thm:mainmulti} we rigorously establish a sufficient criterion for the resulting structure to be isomorphic to the \emph{Rado multigraph} with probability 1. (This structure is the natural multigraph analogue of the Rado graph, and is defined in Definition \ref{defin:axioms} below.) Our criterion is that $f$ is asymptotically bounded above and below by positive non-constant linear functions of $t$.

In \cite{Elw}, the author uses similar machinery to analyse a Preferential Attachment process in which parallel edges are not permitted, and the new node $t+1$ is connected to each pre-existing node $u$ independently with probability $\frac{d_u(t)}{t}$. Thus the number of new edges is not prescribed, but is itself a random variable. It is shown in \cite{Elw} that, so long as the initial graph is neither edgeless nor complete, with probability 1 the infinite limit of the process will be the Rado graph augmented with a finite number of either universal or isolated nodes.

We describe the structure of the remainder of the paper: 

\begin{itemize}
\item In section \ref{section:axioms} we introduce the infinite \emph{Rado Multigraph}.
\item In section \ref{section:PA} we introduce $\textrm{MPA}_f$, our main variant of the preferential attachment process, as well as a secondary variant $\textrm{GPA}_f$. We describe suitable hypotheses on the function $f$, and we prove some initial results. We state our main result, Theorem \ref{thm:mainmulti}, which asserts that under appropriate conditions $\textrm{MPA}_f$ approaches the Rado multigraph.
\item In section \ref{sec:martingale} we develop the theory of martingales for the process $\textrm{MPA}_f$, our main probabilistic tool.
\item In section \ref{sec:finalproof}, we complete the proof of Theorem \ref{thm:mainmulti}. 
\item In section \ref{sec:furtherwork}, we close with some discussion of possible further directions of study.

\end{itemize}

\section{The Rado Multigraph} \label{section:axioms}

We begin by defining the infinite structure which our finitary processes will be shown to approach. So far as we are aware, this structure has not previously appeared in the literature. However the reader familiar with the Rado graph will find little of surprise. (For clarity, we work with the convention that $0 \in \Nat$.)

\begin{defin} \label{defin:axioms}
A \emph{finitary loopless multigraph} is a structure $(V,E)$ where $V$ is a set of vertices, and $E$ is a finitary multiset of unordered pairs from $V$. That is to say every element of $E$ is of the form $e = \{v_i, v_j\}$ (written $v_i v_j$) where $v_i, v_j \in V$ are distinct, $E$ is itself unordered, and each $e$ has a multiplicity $m_e \in \Nat$ describing the number of occurrences of $e$ within $E$. (If $e$ does not occur within $E$ we consider it to have multiplicity $0$.)

\emph{The Rado Multigraph} is a finitary loopless multigraph where $V$ is countably infinite and which additionally satisfies the following axiom:  

\noindent $\bullet$ For any $n \geq 1$, any $m_1, \ldots, m_n \in \Nat$, and any distinct $u_1, \ldots, u_n \in V$ there exists $v \in V$ such that the multiplicity of $v u_i$ is exactly $m_i$.
\end{defin} 
 
We now show that the Rado Multigraph is unique up to isomorphism. We also take the opportunity to observe that several other familiar properties of the Rado graph  hold in our context although we shall not use them directly:

\begin{prop} \label{prop:multiprop}
Let $\mathcal{M}$ and $\mathcal{M}'$ be structures satisfying Definition \ref{defin:axioms} of the Rado Multigraph. Then the following hold:
\begin{enumerate}
\item $\aleph_0$\emph{-categoricity:} $\mathcal{M} \cong \mathcal{M}'$.
\item \emph{1-transitivity:} Given vertices $v_1, v_2$ in $\mathcal{M}$ there exists some $\alpha \in \mathrm{Aut}(\mathcal{M})$ where $\alpha(v_1)=v_2$. 
\item \emph{Ultrahomogeneity:} If $A,B$ are finite substructures of $\mathcal{M}$ and $\gamma:A \cong B$ is a multigraph-isomorphism, then there exists $\alpha \in \mathrm{Aut}(\mathcal{M})$ where $\alpha\upharpoonright_A =\gamma$.

\noindent (Note: here we treat $A$ and $B$ as \emph{induced substructures}: for any vertices $u,v \in A$ the multiplicity of $uv$ within $A$ equals that within $\mathcal{M}$).
\item \emph{Universality:} Any finite or countably infinite finitary loopless multigraph can be isomorphically embedded in $\mathcal{M}$.

\end{enumerate}
\end{prop}

\begin{proof}
We concentrate on proving statement 1. (Statements 2-4 follow from minor alterations to our argument. We leave the reader to fill in the details.) We proceed by a standard back-and-forth argument. First we list the elements of $\mathcal{M}$ as $a_1,a_2,a_3,\ldots$ and similarly $b_1, b_2, b_3, \ldots$ for $\mathcal{M}'$. Now we argue inductively. Suppose $i$ is even, and suppose $(a'_1, \ldots, a'_i ) \cong (b'_1, \ldots, b'_i)$ have been chosen. Let $a'_{i+1}=a_j$ where $j$ is minimum such that $a_j \not\in \{a'_1, \ldots, a'_i \}$.

Let $(m_1, \ldots, m_i)$ be the vector counting the edges between $a'_{i+1}$ and $(a'_1, \ldots, a'_i )$. Notice that each $m_j \in {\bf N}$ by the assumption of finitariness. Then by hypothesis there exists $b'_{i+1}$ joined to $(b'_1, \ldots, b'_i)$ in a fashion described by $(m_1, \ldots, m_i)$. Hence $(a'_1, \ldots, a'_{i+1} ) \cong (b'_1, \ldots, b'_{i+1})$.

Odd steps are identical, exchanging the roles of $\mathcal{M}$ and $\mathcal{M}'$. Thus we build an isomorphism $\mathcal{M} \cong \mathcal{M}'$.\end{proof}
 
Our concern in the current work is on PA processes. However, we remark in passing that the Rado multigraph also arises from the following process in the style of Erd\H{o}s-R\'enyi. We shall not use this result and leave the proof as an easy adaptation of the corresponding classical result about the Rado graph.

\begin{prop}
Let $(p_j)_{j \geq 1}$ be any sequence lying entirely in $(0,1)$. Let $V$ be a countably infinite set. Let $\mathcal{M}$ be multigraph arising from the following random process.

\begin{itemize}
\item At step 0, the structure has no edges.
\item At step $j \geq 1$, consider every pair of distinct $v_1,v_2 \in V$ where $v_iv_j$ currently has multiplicity $j-1$, and connect $v_1,v_2$ with a new $j$\textsuperscript{th} edge with probability $p_{j}$, independently of the behaviour of all other vertices.
\end{itemize}

Then with probability $1$, $\mathcal{M}$ is isomorphic to the Rado multigraph.
 \end{prop}

\section{Preferential Attachment with Prescribed Edge Growth} \label{section:PA}

In this section we shall describe two variants of the preferential attachment process, establish some of their basic properties, and formally state our main result. The process of principal interest will be $\textrm{MPA}_f$ which builds a directed multigraph. We will also mention a natural variant $\textrm{GPA}_f$, which builds a directed graph. Each proceeds by adding, at each time step, a single node along with a prescribed number of directed edges emanating from it. The number of these edges is determined by some fixed function $f:\textbf{N} \to \textbf{N}$. (In fact the directions of the edges will play no role in the theory: we shall analyse the resulting structures as undirected (multi)graphs. However in the interim it will be convenient to refer to the `start-' and `end-points' of each edge, so we preserve directedness for the time being.) We shall work over some initial directed (multi)graph $G'$ containing no isolated nodes (i.e. nodes of degree $0$). However our results will be independent of the choice of $G'$, so the reader may choose to focus on the case where $G'$ is trivial. 

\begin{defin}[The process $\textrm{MPA}_f$] \label{defin:MPA}
Let $G'= (V',E')$ be a finite directed multigraph containing no isolated nodes (so $E'$ is a multiset of ordered pairs from $V'$). Suppose that $G'$ contains $|E'|=e'$ edges and $|V'|=v'$ nodes. We will assume $V'=\{1,\ldots, v'\}$. 

Suppose that the function $f:\textbf{N} \to \textbf{N}$ satisfies:
\begin{itemize}
\item $f(0)=e'$.
\item $f(t)=0$ whenever $1 \leq t \leq v'-1$.
\item $f(t) \geq 1$ for all $t \geq v'$.
\end{itemize}

At each time-step $t \geq 1$, we shall construct a multigraph $G(t)$ with vertex set $V(t)$ and edge multiset $E(t)$.

First we impose $G(1) = \ldots  = G(v') = G'$.

Whenever $t \geq v'$, we will have $V(t):=\{1,\ldots,t\}$ and
$$E(t+1) =E(t) \cup \mathcal{E}(t+1)$$ where $|\mathcal{E} (t+1)| = f(t)$.

The start-point of each edge in $\mathcal{E} (t+1)$ is the new node $t+1$. The end-points are chosen independently from $V(t)$, with probabilities directly proportional to their degrees in $G(t)$. 
\end{defin}

Notice that, the degrees used to calculate the probabilities are taken from $G(t)$, which is to say the model does not notice any incremental updating of degrees between $G(t)$ and $G(t+1)$. One can imagine the endpoints of the $f(t)$ many new edges being selected simultaneously, and independently of each other.

Notice too that our assumption that $f(t) \neq 0$ for $t \geq v'$ (along with our assumption on $G'$) serves to ensure that there are never any isolated nodes.

We may now state our main result. (Recall the asymptotic notation $g_1 = \Theta (g_2)$ for functions $g_1, g_2$ as meaning that there exist $c_2 \geq c_1>0$ so that for all large enough $t$ we have $c_1 \cdot g_2(t) \leq g_1(t) \leq c_2 \cdot g_2(t)$.)

\begin{restatable}{theorem}{mainmulti}  \label{thm:mainmulti}
Suppose that $G'$ is a finite directed multigraph containing no isolated nodes, that $f$ satisfies the requirements from Definition \ref{defin:MPA}, and also that $f(t) = \Theta(t)$. Then, with probability 1, the infinite limit of $\MPA(G')$ is isomorphic, as an undirected multigraph, to the Rado multigraph.
\end{restatable}

Before we commence the proof of this theorem, we remark that we expect that a similar result to apply to a graph variant of the process, which we briefly introduce:

\begin{defin}[The process $\textrm{GPA}_f$] \label{defin:GPA}
Let $G'= (V',E')$ be a finite directed graph containing no isolated nodes. Suppose that $G'$ contains $|E'|=e'$ edges and $|V'|=v'$ nodes. We will assume $V'=\{1,\ldots, v'\}$. 

Suppose that the function $f:\textbf{N} \to \textbf{N}$ satisfies:
\begin{itemize}
\item $f(0)=e'$.
\item $f(t)=0$ whenever $1 \leq t \leq v'-1$.
\item $1 \leq f(t) \leq t$ for all $t \geq v'$.
\end{itemize}

At each time-step $t \geq 1$, we shall construct a graph $G(t)$ with vertex set $V(t)$ and edge set $E(t)$.

First we impose $G(1) = \ldots  = G(v') = G'$.

Whenever $t \geq v'$, we will have $V(t):=\{1,\ldots,t\}$ and
$$E(t+1) =E(t) \cup \mathcal{E}(t+1)$$ where $|\mathcal{E} (t+1)| = f(t)$.

The start-point of each edge in $\mathcal{E} (t+1)$ is the new node $t+1$. The end-points of the edges are selected sequentially from $V(t)$, without replacement, with the choice at each step made from the remaining unselected elements of $V(t)$ with probabilities directly proportional to their degrees in $G(t)$. 
\end{defin}

\begin{restatable}{conj}{mainconj} \label{conj:mainconj}
Suppose that $G'= (V',E')$ be a finite directed graph containing no isolated nodes, that $f$ satisfies the conditions in Definition \ref{defin:GPA}, and further that there are constants $0 < c_1 \leq c_2 <1$ where $c_1 \cdot t \leq f(t) \leq c_2 \cdot t$ for all large enough $t$. Then, with probability 1, the infinite limit of $\GPA(G')$ is isomorphic as an undirected graph to the Rado graph.
\end{restatable}

Our arguments will be independent of $G'$, and thus we shall largely suppress mention of it.  Let us now consider the distribution of edges at stage $t+1$. First notice that $|E (t)| = F(t):=\sum_{i=0}^{t-1} f(i)$. Hence in $\textrm{MPA}_f$, at stage $t+1$ given any pre-existing node $u \leq t$, the probability that any given edge in $\mathcal{E} (t+1)$ has its end-point at $u$ is exactly $\frac{d_u(t)}{2F(t)}$, where $d_u(t)$ is the degree of $u$ in $G(t)$. Thus the expected number of edges in $\mathcal{E}(t+1)$ with endpoint at $u$ is $\frac{f(t) \cdot d_u(t)}{2F(t)}$.

In $\textrm{GPA}_f$ this probability distribution is more complicated, and the expected number of edges $u$ receives at stage $t+1$ depends in a more detailed way upon $G(t)$. This is the primary obstacle to extending the current work to a proof of Conjecture \ref{conj:mainconj}.

Our standing assumption will be that we are working in $\MPA$. We shall leave the case of $\GPA$ open, but make some remarks about it as we proceed.

Our assumption in Theorem \ref{thm:mainmulti} is that $f(t)= \Theta(t)$. However we shall be able to develop much of the theory under the following weaker hypotheses:

\begin{assumption} \label{assumption:ass}
\begin{equation}  \sum_{s=0}^\infty \frac{f(s)}{F(s)} = \infty. \end{equation}
\begin{equation}  \sum_{s=0}^\infty \left( \frac{f(s)}{F(s)} \right)^2 < \infty.\end{equation}
\end{assumption}

We briefly discuss this. Assumption \ref{assumption:ass} easily follows in full, for instance, if $f(t)= \Theta(t^{\alpha})$ for some $\alpha \geq 0$.

However part (2) fails in general for polynomially bounded functions, an example being: 
$$f(t) = \left\{ \begin{array}{ll} t & \textrm{ when } t=2^n \textrm{ for } n \in \Nat\\
1 & \textrm{ otherwise.} \end{array} \right. $$ 

On the other hand, both parts do hold for some exponential functions, such as $f(t)=\lfloor \frac{1}{4} t^{- \frac{3}{4}} e^{\frac{1}{4}t} \rfloor$.



In all cases, it will be useful to extend the domain of $f$ to ${\bf R}^{\geq 0}$. We choose to do this as a step function, via $f(t):=f \left( \lfloor t \rfloor \right)$. (Of course there may be more natural ways to achieve the same thing, however this choice will be convenient, as the fourth point in the following Lemma makes clear.) We now gather together some observations about the extended function $f$. These follow immediately from our previous conditions.

\begin{coro} 
The following hold:
\begin{itemize}
\item $f(t)=e'$ for $0 \leq t <1$.
\item $f(t)=0$ whenever $1 \leq t < v'$.
\item $f(t) \geq 1$ for all $t \geq v'$.
\item $f$ is Lebesgue-measurable with antiderivative $\ds \int_{0}^{t} f(s) ds =:F(t)$. (This notation is consistent with the previous interpretation of $F$ since the two functions coincide at integer points.)
\item $F$ is monotonic increasing everywhere and strictly so for $t \geq v'$.
\end{itemize}
\end{coro}

Under our additional hypothesis we can say a little more:

\begin{lem} \label{lem:integral}

Suppose that Assumption \ref{assumption:ass}(2) holds. Then for any $\beta \geq 1$, there exists $K_\beta >0$ so that for any $t \geq m \geq 0$:
$$\left| \int_m^t \frac{f(s)}{F(s)^\beta} ds - \sum_{s=m}^{t} \frac{f(s)}{F(s)^\beta} \right| < K_\beta.$$
\end{lem}

\begin{proof}

Let $M$ be such that whenever $s \geq M$ then $f(s) < F(s)$. Such a value must exist by Assumption \ref{assumption:ass}(2). 

It is enough to prove the result this for all $m \geq M$, since one can then add $\ds \max \left\{ \int_0^M \frac{f(s)}{F(s)^\beta} ds,\  \sum_{s=0}^{M} \frac{f(s)}{F(s)^\beta} \right\}$ to $K_{\beta}$ to obtain the result for all $m$. Thus we shall assume $m \geq M$.

Firstly, it is immediate by consideration of $F\upharpoonright_{[s,s+1]}$ that 
$$\sum_{s=m}^{t} \frac{f(s)}{F(s+1)^\beta} < \int_m^t \frac{f(s)}{F(s)^\beta} ds < \sum_{s=m}^{t} \frac{f(s)}{F(s)^\beta}.$$
Next we shall appeal to Newton's generalised binomial theorem, that whenever $a,b,\beta \in \bf{C}$ with $0<|b|<|a|$, then $\ds (a+b)^{\beta} = \sum_{j=0}^{\infty} C(\beta, j) a^{\beta - j} b^j,$ where $C(\beta,j)$ are the generalised binomial coefficients.

When $a=1$, the series has radius of convergence $1$ in $b$. We shall also use the fact that the series remains convergent for $|b|=1$, so long as $\textrm{Re}(\beta)>0$, which of course holds in the context of this Lemma. (See \cite{CKP} p.17, for example.) Now,

\begin{align*}
\sum_{s=m}^{t} \frac{f(s)}{F(s)^\beta} &- \sum_{s=m}^{t} \frac{f(s)}{F(s+1)^\beta} = 
\sum_{s=m}^{t} \frac{f(s)}{F(s)^\beta} - \frac{f(s)}{\left(F(s)+ f(s)\right)^\beta} \\
& < \sum_{s=m}^{t} \frac{f(s)\left(F(s)+ f(s)\right)^\beta - f(s) F(s)^\beta}{F(s)^{2\beta}}\\
& = \sum_{s=m}^{t} \frac{f(s)\left(\sum_{j=1}^{\infty} C(\beta,j) F(s)^{\beta - j}f(s)^j \right)}{F(s)^{2\beta}}\\
& < \sum_{s=m}^{t} \frac{\sum_{j=1}^{\infty} C(\beta,j) F(s)^{\beta - 1}f(s)^2}{F(s)^{2 \beta}}\\
& < \sum_{s=m}^{t} \frac{2^\beta f(s)^2}{F(s)^{1 + \beta}} \ \leq  \ 2^\beta \sum_{s=m}^{t} \frac{ f(s)^2}{F(s)^{2}} < 2^\beta \cdot K := K_\beta
\end{align*}
where $K$ is the finite bound provided in Assumption \ref{assumption:ass} (2).

\end{proof}

The next two results hold in $\GPA$ as well as $\MPA$:

\begin{lem} \label{lem:nonzero}
Suppose that Assumption \ref{assumption:ass}(1) holds. Then for any node $u$, any stage $t_0$, and any state of the graph $G(t_0)$, the probability that $v$ never receives another edge is $0$.
\end{lem}

\begin{proof}

Suppose that $d_u(t_0)=N \geq 1$. The probability that $u$ never receives a further edge is therefore given by (or in $\GPA$ is bounded above by) $$\prod_{t=t_0}^{\infty} \left(1 - \frac{N}{2F(t)} \right)^{f(t)}.$$

We shall show that this is $0$. It is clearly enough to do so in the case $N=1$. Taking logarithms, it is therefore enough to show that
$$\sum_{t=t_0}^{\infty} f(t) \ln \left(1 + \frac{1}{2F(t) - 1}  \right)$$ diverges to $\infty$.
Now as for small enough $x$, we know $\ln(1+x)> \frac{1}{2}x$. Thus for large enough $t$, 
$$\ln \left(1 + \frac{1}{2F(t) - 1} \right) > \frac{1}{4 F(t)}.$$
Thus the result follows from Assumption \ref{assumption:ass}(1).
\end{proof}

\begin{coro} \label{coro:dinf}
Suppose that Assumption \ref{assumption:ass}(1) holds. Then for any node $u$, given any state of the graph $G(t_0)$, with probability $1$ it will be true that $d(t) \to \infty$ as $t \to \infty$.
\end{coro}

\begin{proof}
This follows automatically from Lemma \ref{lem:nonzero} by the countable additivity of the probability measure.
\end{proof}

\section{Martingale Theory}\label{sec:martingale}

In this section, we apply some machinery from the theory of martingales to the process $\MPA$, generalising the theory developed in \cite{KK}. We shall assume throughout that we are working in $\MPA$, and begin with the following easy result, which does not transfer immediately to $\GPA$.

\begin{rem} \label{rem:probbound}
Given any node $u$, define $U_u(t+1):= d_u(t+1) - d_u(t)$ and
$\mu_u(t):= \E \left( U_u(t+1) \big| \big| d_u(t) \right).$ Then $$\frac{\mu_u(t)}{d_u(t)} = \frac{f(t)}{2F(t)}.$$
In particular, if $f(t) = \Theta \left( t^{\alpha} \right)$ where $\alpha \geq 0 $ then $\ds \mu_u(t) = \Theta \left( \frac{d_u(t)}{t} \right).$ 
\end{rem}

The next two results are the key to our analysis, and generalise Proposition 3.1 of \cite{KK}

\begin{prop} \label{prop:xuexists}

Suppose that Assumption \ref{assumption:ass} (2) holds. For any node $u$, define $$A(t) = A_u(t):=\prod_{j=1}^{t-1} \left( 1 + \frac{f(j)}{2F(j)}\right)$$ and
$X(t):= X_u(t) = \frac{d_u(t)}{A_u(t)}$. Then 

\begin{enumerate}[(i)]
\item $X(t)$ is a martingale.
\item Thus, for any node $u$, with probability 1, there exists $x_u \geq 0$ such that $$\lim_{t \to \infty} \frac{d_u(t)}{A(t)} = x_u.$$
\item $A(t) = \Theta \left( F(t)^{\frac{1}{2}} \right)$
\end{enumerate}
\end{prop}

\begin{proof}
Employing Remark \ref{rem:probbound}, the first part is straightforward:
\begin{eqnarray*} \E (X(t+1) || X(t)) &=& \frac{1}{A(t+1)} \E (d(t+1)||d(t))\\
&=& \frac{1}{A(t+1)}\left( d(t) +\mu(t) \right)\\
&=& \frac{1}{A(t+1)}d(t) \left( 1 +\frac{\mu(t)}{d(t)} \right)\\
&=&  \frac{1}{A(t)} d(t) = X(t) .\end{eqnarray*}

Part (ii) follows from (i) via Doob's convergence theorem, which gives us that $X(t) \to X$ for some random variable $X$.

Hence all that remains to understand the behaviour of $A(t)$ for large $t$, to establish (iii). By taking logarithms and employing the standard bounds $x- \frac{1}{2}x^2 < \ln(1+x)<x$, we see:
$$\frac{1}{2} \sum_{s=1}^{t-1} \frac{f(s)}{F(s)} - \frac{1}{8} \sum_{s=1}^{t-1} \frac{f(s)^2}{F(s)^2} < \ln A(t)< \frac{1}{2} \sum_{s=1}^{t-1} \frac{f(s)}{F(s)}.$$
Therefore by Assumption \ref{assumption:ass} and Lemma \ref{lem:integral}, it follows that
$$\frac{1}{2} \int_{s=1}^{t-1} \frac{f(s)}{F(s)}ds - K < \ln A(t)<\frac{1}{2} \int_{s=1}^{t-1} \frac{f(s)}{F(s)}ds + K'$$
for some constants $K$ and $K'$ from which the result follows.
\end{proof}

We need a little more information about the distribution of the $x_u$ provided by the preceding result:

\begin{prop} \label{prop:poslimit}
Suppose that $f$ satisfies Assumption \ref{assumption:ass} in full. Given any time $t_0$, state $G_0(t_0)$, and node $u$, $\Prob \left(x_u>0 \right) = 1$.
\end{prop}

\begin{proof}
Our proof closely follows that of Proposition 3.1 of \cite{KK}.

We take $u$ as fixed and shall suppress mention of it, writing $X(n)$ for $X_u(n)$, etc., throughout.

Given any $n>m>0$ define $\tilde{X}_{m}(n) := \left( X(n) - X(m) \right)^2$. Then for fixed $m$, it is an elementary fact that the sequence $\tilde{X}_{m}(n)$ forms a submartingale. We now proceed via a sequence of claims.\\

\noindent {\bf Claim 1} $$\E \left( \tilde{X}_{m}(n) \big|\big| X(m) \right) = \sum_{t=m}^{n-1} \E \Big( X(t+1)^2 \big|\big| X(m) \Big) - \E \Big( X(t)^2 \big|\big| X(m) \Big).$$
\noindent {\bf Proof of Claim 1 \ \ } 
\begin{eqnarray*}
\E \Big( \tilde{X}_{m}(n) & \big|\big| & X(m) \Big) \\
&=& \E \Big( X(n)^2 - 2X(n) X(m) + X(m)^2 \big|\big| X(m) \Big)\\
& = & \E \Big( X(n)^2  \big|\big| X(m) \Big) - 2 X(m) \E \Big( X(n) \big|\big| X(m) \Big) + X(m)^2\\
& = & \E \Big( X(n)^2 \big|\big| X(m)\Big) - X(m)^2.
\end{eqnarray*}
Unpacking the sum in the statement of the claim gives the same result. \textbf{QED Claim 1}\\

\noindent {\bf Claim 2 \ \ } There exists $K>0$ such that for all large enough $m$ and all $n>m$ $$\E (\tilde{X}_{m}(n) ||X(m)) < X(m) \cdot \frac{K}{F(m)^{\frac{1}{2}}}.$$
\noindent {\bf Proof of Claim 2 \ \ } Recall $U(t+1) := d(t+1) - d(t)$. Now $U(t+1)$ is binomially distributed via $b \left(f(t), \frac{d(t)}{2F(t)}  \right)$ meaning, as already observed, that $\E (U(t+1) || d(t) = d) = \mu(t) = \frac{d \cdot f(t)}{2F(t)}$ and also $\V(U(t+1)||d(t)=d) = \frac{d \cdot f(t)}{2F(t)}\left( 1- \frac{d}{2F(t)}\right)$. Thus, writing $f$ and $F$ for $f(t)$ and $F(t)$ respectively,
\begin{eqnarray*} \label{eqnarray:Z2}
\E \Big( U(t+1)^2 ||d(t)=d \Big) & = & \left(\frac{df}{2F} \right)^2 + \frac{df}{2F}\left(\frac{2F-d}{2F} \right)\\
& < & \frac{fd}{2F} + \frac{f^2d^2}{4F^2}.
\end{eqnarray*}
At the same time,
\begin{align*}
\E & \Big( d(t+1)^2 || d(t)=  d \Big) \\
& = \ \E \Big( \left(U(t+1) +d \right)^2||d(t)=d \Big)\\
& =  \ \E (U(t+1)^2||d(t)=d) + 2d \E (U(t+1) || d(t)=d) + d^2\\
& <  \ \frac{fd}{2F} + \frac{f^2d^2}{4F^2} +2d \cdot \frac{df}{2F} + d^2\\
& =   \ \frac{fd}{2F} + \left(1+ \frac{f}{2F} \right)^2 d^2.
\end{align*}

Recall the definition of the martingale $X(t):= \frac{d(t)}{A(t)}$. Thus
\begin{eqnarray*}
\E \Big( X(t+1)^2 & \ \Big| \Big| & d(t)=d \Big) = \left( \frac{1}{A(t+1)^2} \right) \cdot \E \Big( d(t+1)^2\ \Big| \Big | \ d(t)=d \Big)\\
& < & \frac{1}{A(t+1)^2} \left(\frac{fd}{2F} + \left(1+ \frac{f}{2F} \right)^2 d^2 \right)\\
& = & \frac{f A(t)}{2F \cdot A(t+1)^2} X(t) + \left(1+ \frac{f}{2F} \right)^2 \left( \frac{A(t)}{A(t+1)}\right)^2 X(t)^2\\
& < & \frac{f}{2F \cdot A(t)} X(t) + X(t)^2.
\end{eqnarray*}
Hence, by the law of total expectation,
$$\E (X(t+1)^2|| X(m)) - \E (X(t)^2||X(m)) < \frac{f}{2F \cdot A(t)} X(m).$$
Summing this up over successive terms (and appealing to Claim 1, Proposition \ref{prop:xuexists} (ii) and Lemma \ref{lem:integral}) we get 
\begin{eqnarray*}
\E (\tilde{X}_{m}(n) &||& X(m)) < X(m) \cdot \sum_{t=m}^{n-1} \frac{f}{2F A(t)}\\
& < & X(m) \cdot  \sum_{t=m}^{n-1} \frac{f}{2F A(t)}\\
& = & O \left( X(m) \cdot \sum_{t=m}^{n-1} \frac{f}{F^{\frac{3}{2}}} \right)\\
& = & O \left( X(m) \cdot \int_{t=m}^{n-1} \frac{f(t)}{F(t)^{\frac{3}{2}}} dt \right)\\
& < & X(m) \cdot \frac{K}{F(m)^{\frac{1}{2}}} \textrm{ for some $K>0$. \textbf{QED Claim 2.} }\\
\end{eqnarray*}

We may now prove the proposition. We proceed by defining a sequence of times: $n_0 = t_0$. Let $n_{i+1}$ be the least $n$ (if any exists) such that $X(n) < \frac{1}{2}X \left(n_{i} \right)$. Otherwise $n_{i+1} = \infty$.

The trick is to apply the Kolmogorov-Doob inequality (see for instance \cite{KK}) to $\tilde{X}_{n_i}(n)$:

\begin{eqnarray*}
\Prob ( n_{i+1} < \infty ||n_{i} < \infty) & = & \Prob \left(  \min_{n \geq n_i} X(n) < \frac{1}{2} X(n_i) {\Big|\Big|} X(n_i)\right)\\
& \leq & \Prob \left(  \max_{n \geq n_i} \tilde{X}_{n_i}(n) > \frac{1}{4} X(n_i)^2 {\Big|\Big|} X(n_i)\right) \\
& = & \lim_{N\to \infty} \Prob \left(  \max_{n: N \geq n \geq n_i} \tilde{X}_{n_i}(n) > \frac{1}{4} X(n_i)^2 {\Big|\Big|} X(n_i)\right) \\
& \leq & \frac{4}{X(n_i)^2} \cdot \lim_{N \to \infty} \E (\tilde{X}_{n_i}(N) || X(n_i))\\
&  =  & O \left( \frac{4}{X(n_i)^2} \cdot \frac{1}{F(n_i)^{\frac{1}{2}}} \cdot X(n_i) \right) \\
& = & O \left( \frac{1}{d(n_i)} \right). 
\end{eqnarray*}

It follows from Corollary \ref{coro:dinf} that $\Prob ( n_{i+1} < \infty ||n_{i} < \infty )\to 0$ as $i \to \infty$, from which the result follows. \end{proof}

We record one more result regarding the martingale $X(t)$: 
\begin{coro} \label{coro:L2}
Suppose that Assumption \ref{assumption:ass} (2) holds. Then the martingale $X(t)$ is bounded in $\mathcal{L}_2$, that is to say $\sup_{t} \E \left( X(t)^2 \right) < \infty$. 
\end{coro}

\begin{proof}
By a standard result (see for example Theorem 12.1 of \cite{Will}), it is sufficient to show that $\sum_{j=0}^{\infty} \E \left(|X_{j+1} - X_j|^2 \right) < \infty$.

Notice that by Remark \ref{rem:probbound}
\begin{align*}
|X(t+1) - X(t)| & = \frac{d(t+1) - \left(1 + \frac{\mu(t)}{d(t)} \right)d(t)}{A(t+1)}\\
& = \frac{U(t+1)- \mu(t)}{A(t+1)}
\end{align*}

Hence
\begin{align*}
\E \left(|X(t+1) - X(t)|^2 \right)  & = \frac{\V(U(t+1) )}{A(t+1)^2}\\
&   = O \left( \frac{d(t) f(t)}{2F(t)}\left( 1- \frac{d(t)}{2F(t)}\right) \cdot \frac{1}{A(t+1)^2} \right)\\
& =  O \left( \frac{d(t)}{A(t+1)} \cdot \frac{f(t)}{F(t) A(t+1)} \right)\\
&  = O \left( X(t) \cdot \frac{f(t)}{F(t) A(t+1)} \right)\\
&  = O \left( \frac{f(t)}{F(t) A(t+1)} \right)\\
&   =O \left( \frac{f(t)}{F(t)^{\frac{3}{2}}} \right).
\end{align*}
Thus \begin{align*}
\sum_{j=0}^{t} \E \left(|X_{j+1} - X_j|^2 \right)  & = O \left( \int_{j=0}^t \frac{f(j)}{F(j)^{\frac{3}{2}}} dj \right)\\
 & = O \left( K - F(t)^{-\frac{1}{2}} \right) =O(K).
\end{align*}
\end{proof}

\section{Proof of Main result} \label{sec:finalproof}

\begin{defin}
A \emph{witness request} $W$ is a set of pairs of the form $W=\{(u_i,m_i) \ | \  1 \leq i \leq n \}$ where $(u_1, \ldots, u_n)$ is a sequence of nodes and $(m_1, \ldots, m_n)$ an accompanying sequence of non-negative integers.

A \emph{witness} for $W$ is a node connected to each $u_i$ with multiplicity $m_i$. 

We write the event $W[t]$ to mean that $W$ is satisfied by some witness by time $t$.
\end{defin}

Observe from the structure of the process that $W[t] \Rightarrow W[t']$ for all $t' \geq t$. The following is the major step towards our goal:

\begin{prop} \label{prop:mainstep1}
Suppose that $f(t) = \Theta(t)$, and that $G(t_0)$ is a state of the graph at time $t_0$. Let $W$ be a witness request. Let $\eps>0$. Then there exist $t_1 > t_0$ such that $\Prob \Big(W[t_1] \ \big| \big| \ G(t_0) \Big) >1 - \eps$.   
\end{prop}

\begin{proof}
We consider only stages from $t_0+1$ onwards, and everything that occurs is conditioned upon $G(t_0)$, which we shall therefore suppress.

Suppose $W=\{(u_i,m_i) \ | \  1 \leq i \leq n \}$. We shall write $m=\sum_{i=1}^n m_i$, and, abusing notation, $U_i=U_{u_i}(t+1)$, meaning the number of new edges which $u_i$ gains at the $t+1$st stage, taking the dependency on $t$ as given when the intended value is obvious. Similarly we write $d_i$ for $d_{u_i}(t)$. (We shall not consider $d_j$ for any $j$ other than the $u_i$, so this will not cause confusion.)

We shall employ vector notation, writing $\textbf{U}(t+1):=\textbf{U}=(U_1,\ldots,U_n)$ and $\textbf{m}:=(m_1,\ldots,m_n)$. Thus our focus is the event $\textbf{U}=\textbf{m}$. Let us first compute the probability of this event in terms of the $d_i$. The relevant distribution is multinomial $M(f,p_i, \ldots, p_n,q)$ where $p_i = \frac{d_i}{2F}$ and $q = 1 - \sum p_i$ (again omitting the dependencies on $t$). Therefore
\begin{align*}
\Prob &\left(\textbf{U}=\textbf{m} \ \big| \big| \  d_1, \ldots d_n \right) \\
 & = \frac{f!}{m_1!\cdot \ldots m_n! \cdot (f - m)!} \cdot q^{f-m} \cdot \prod_{i} {p_i}^{m_i}\\
&= \Theta \left( \left( 1-\frac{\sum_i d_i}{2F} \right)^{f-m} \cdot \left( \frac{f}{2F} \right)^m \cdot \prod_{i} d_i^{\ m_i} \right)
\end{align*}
noticing that $\frac{f!}{(f - m)!} \sim f^{m}$. 

Now we employ our assumption that $f(t) = \Theta (t)$ from which it also follows that $\frac{1}{2F} = \Theta\left( \frac{1}{t^2} \right) $ and $\frac{f}{2F} = \Theta\left( \frac{1}{t} \right) $. Thus there exist constants  $c_1, c_2, C_0, N>0$ depending only on $G_0(t_0)$ such that for all $t \geq N$,
\begin{equation} \label{equation:firstbound}\Prob \left(\textbf{U}=\textbf{m} \ \big| \big| \  d_1, \ldots d_n \right) \geq C_0 \cdot \left( 1-\frac{\sum_i d_i}{c_1 t^2} \right)^{c_2 t-m} \cdot t^{- m} \cdot \prod_{i} d_i^{\ m_i}. \end{equation}

Our aim is to bound this probability below, away from $0$ over a long enough range of $t$. We write $X_i = \frac{d_i}{A_i}$ for the martingale supplied by Proposition \ref{prop:xuexists}, with $x_i:=x_{u_i}>0$ for  its limit supplied by Proposition \ref{prop:xuexists} and Lemma \ref{prop:poslimit}. We will not attempt to condition on the actual values $x_i$, but only on the fact that these values are not extreme ($\textbf{NE}$).

First, choose $\kappa_1, \kappa_2 >0$ such that
\begin{equation*}
\kappa_1 t <A(t) < \kappa_2 t
\end{equation*}
for all large enough $t$. This is guaranteed to occur by Proposition \ref{prop:xuexists} (iii) since $F(t)^{\frac{1}{2}}  = \Theta (t)$. We increase $N$ if necessary to ensure that this holds. Notice that since $A(t)$ is entirely predictable in advance, the value of $N$ remains dependent only on $G_0(t_0)$. 

Now, for any $y_2>y_1>0$, define the following event:
$$\textbf{NE} (y_1,y_2): \ \ \ \left(  \bigwedge_{i=1}^n y_1 < x_i < y_2 \right).$$

We shall apply this in the following case: given $\delta>0$ choose $y_2(\delta)> y_1(\delta)>0$ so that $\Prob(\neg \textbf{NE} (y_1,y_2)) < \delta$. (We shall specify $\delta$ later, and will only need to consider one such value. Thus we shall consider $\delta$ fixed for the purposes of what follows.)

By Corollary \ref{coro:L2},  $X_i(t) \to x_i$ in expectation, and thus in probability. More precisely, for any $\eta>0$, we may increase $N>0$ by some quantity depending only on $\eta$ so that for all $t \geq N$ and all $i \leq n$
$$\E \Big( |X_i(t) - x_i| \Big) < \left(\frac{\eta}{n} \right)^2.$$ 
Thus, by Markov's inequality
   $$\Prob\left(|X_i(t) - x_i| > \frac{\eta}{n} \right) < \frac{\eta}{n}.$$
Hence defining the event that all the $X_i(t)$ are close ({\bf Cl}) to their respective $x_i$
\begin{equation*} \textbf{Cl}(t, \eta):= \ \ \bigwedge_{i=1}^n \left( |X_i(t) - x_i| < \frac{\eta}{n} \right)\end{equation*}
we have for all $t \geq N$
\begin{equation} \label{equation:zbound} \Prob\left(\textbf{Cl} (t, \eta) \right) >1- \eta.\end{equation}
Again, we shall pick a value of $\eta$ later. Notice also that
\begin{equation*}
\Prob\left(\textbf{Cl} (t, \eta) \right) < \Prob\left(\textbf{Cl} (t, \eta) \ \big| \big| \ 
\textbf{NE} (y_1,y_2) \right) + \delta.
\end{equation*}
So \begin{equation}
\Prob\left(\textbf{Cl} (t, \eta) \ \big| \big| \ \textbf{NE} (y_1,y_2) \right) > 1 - \eta - \delta.
\end{equation}

Next, we define a bound for $d_i(t)$. Given $\delta,\eta >0$ as before, let $b_1(\eta)=b_1(\delta,\eta) := \kappa_1 \cdot \left( y_1 - \frac{\eta}{n} \right)$ and $b_2(\eta)= b_2(\delta, \eta) := \kappa_2 \cdot \left( y_2 + \frac{\eta}{n} \right)$, insisting that $\eta$ is small enough that $b_1>0$. Then we define the event
$$\textbf{Bo}(t, b_1, b_2)  :=  \ \ \ \bigwedge_{i=1}^n \left( b_1 \cdot t < d_i(t) < b_2 \cdot t \right).$$
Observe now that for $t \geq N$ \begin{equation} \label{equation:implic} \left( \textbf{NE} (y_1(\delta),y_2(\delta)) \ \& \ \textbf{Cl}(t, \eta) \right) \Rightarrow \textbf{Bo}(t, b_1(\eta), b_2(\eta)). \end{equation}
Hence
$$\Prob \left( \textbf{Bo}(t, b_1, b_2) \Big|\Big| \textbf{NE} (y_1(\delta),y_2(\delta)) \right) \geq 1 - \eta - \delta.$$
 Thus we obtain the unconditional bound:
\begin{equation} \label{equation:uncon}
\Prob \left( \textbf{Bo}(t, b_1, b_2) \right) \geq (1 - \eta - \delta)(1 - \delta).
\end{equation}

Now, we use the bound obtained in (\ref{equation:firstbound}) above, and see that whenever $b_1 \leq b_1' < b_2' \leq b_2$ 
\begin{eqnarray*}
\Prob \Big( \textbf{U}=\textbf{m}  \ & \Big| \Big| & \ \textbf{Bo}(t,b_1',b_2') \Big)\\
& > & C_0 \cdot  \left( 1-\frac{n \cdot b_2 \cdot t}{c_1 t^2} \right)^{c_2 t-m} \cdot t^{-m}\cdot \left(b_1\cdot t\right)^m\\
& = & C_0 \cdot b_1^m \cdot \left( 1-\frac{nb_2}{c_1 t} \right)^{c_2 t-m}\\
& = & C_0 \cdot b_1^m \cdot \left( 1-\frac{nb_2c_2}{c_1}\cdot \frac{1}{c_2 t} \right)^{c_2 t} \cdot \left( 1-\frac{nb_2}{c_1t} \right)^{-m}\\
& \to & C_0 \cdot b_1^m \cdot e^{-\frac{nb_2c_2}{c_1}}:=C_3>0.
\end{eqnarray*}

Hence, by letting $C_4 = C_4 (\delta, \eta):=\frac{1}{2}C_3$ and increasing $N$ again if necessary (and again by some predictable amount), we have for all $t \geq N$
\begin{equation} \label{equation:Ubound}
\Prob\left( \textbf{U}=\textbf{m}  \ \Big| \Big| \ \textbf{Bo}(t,b_1',b_2') \right)>C_4.\end{equation}

Now for any $\zeta>0$, we may let $M=M(\zeta,\delta,\eta)$ be large enough that $\left(1 - C_4\right)^M < \zeta$. The goal therefore is to locate $M$ places where $\textbf{Bo}(t,b_1(\eta),b_2(\eta))$ holds, and argue that the probability that all of them fail to produce an instance of $\textbf{U}=\textbf{m}$ is bounded above by $\zeta$.

Notice that bound (\ref{equation:Ubound}) holds independently for all $t \geq N$: the arguments are unaffected by previous values of $\textbf{U}$ so long as $\textbf{Bo}(t, b_1', b_2')$ holds. However, the same is not true for bound (\ref{equation:uncon}). By conditioning on whether or not $\textbf{U}(t')=\textbf{m}$ holds, we risk affecting $\Prob\left(\textbf{Bo}(t, b_1(\eta), b_2(\eta))\right)$ for $t > t'$.

To navigate this obstacle, we shall locate a range $[t_2, t_2+M)$ within which the bound $\textbf{Bo}(t,b_1(\eta),b_2(\eta))$ is guaranteed to hold, barring a certain extreme event $\neg \textbf{Sh}$ defined below, which will have a  probability bounded above by $\theta$ for arbitrarily small $\theta$.

We wish $t_2$ to satisfy the tighter bound $\textbf{Bo}(t_2,b_1 ( \tfrac{\eta}{2}),b_2 ( \tfrac{\eta}{2}))$. Notice that appropriate adaptations of (\ref{equation:zbound}), (\ref{equation:implic}), and (\ref{equation:uncon}) above guarantee that for large enough $t_2$,
\begin{equation} \label{equation:b2bound}
\Prob \left(\neg \textbf{Bo} \left( t_2,b_1 \! \left( \tfrac{\eta}{2} \right),b_2 \! \left( \tfrac{\eta}{2} \right) \right) \right) < \! \tfrac{\eta}{2} + 2 \delta.
\end{equation}

However, as already indicated, $\textbf{Bo} \left( t_2,b_1 \! \left( \tfrac{\eta}{2} \right),b_2 \! \left( \tfrac{\eta}{2} \right) \right)$ on its own is not quite enough to guarantee $\textbf{Bo}(t_2 + j ,b_1(\eta),b_2(\eta))$ for $j \leq M$. So let us describe the extra ingredient we require. Notice that $U_i(t)$ is a binomial distribution with a long right tail, since the number of trails $f(t)$ is of the order of $t$, and the probability of success per trial is $\frac{d_i(t)}{2F(t)}$ which is of order $\frac{1}{t}$. We shall show that we may ignore the extremity of this tail, thus allowing us to impose a tighter upper bound than $f(t)$ on $U_i(t)$ for all $t \in [t_2, t_2+M)$.

In Theorem 1.1 from \cite{Bol}, we find a useful bound for the right-tail of a binomial distribution $U \sim b(f,p)$: if $u>1$ and $1 \leq S: = \lceil ufp \rceil \leq f-1$ then 
$$\Prob(U \geq S) < \left(\frac{u}{u-1} \right) \cdot \Prob(U = S).$$

Let us apply this in the case $S = \lceil t^{\alpha} \rceil$ for some fixed $\alpha \in \left( \frac{1}{2}, 1 \right)$. (Its exact value does not matter.) Then 
$$u = u(t)  = \frac{t^{\alpha}}{ p f} = \frac{t^{\alpha} 2 F(t)}{d(t) f(t)}.$$

Assembling the bounds $c_1 t \leq f(t) \leq c_2 t$ and $c_1 t^2 \leq 2F(t) \leq c_2 t^2$ and $\textbf{Bo}(t, b_1', b_2')$ where $b_1 \leq b_1' < b_2' \leq b_2$ and  employing the standard bound for the binomial coefficiant $\begin{pmatrix} f \\ S \end{pmatrix} \leq \left( \dfrac{f\cdot e}{S} \right)^S$, we find

\begin{align*}
\Prob \Big(U_i  & \geq t^{\alpha} \  \Big| \Big| \ \textbf{Bo}(t, b_1', b_2') \Big)\\
 & < \left(\frac{\frac{c_2}{c_1 b_1} t^{\alpha}}{\frac{c_1}{c_2 b_2} t^{\alpha}-1} \right) \cdot \left( (e c_2 +1) \cdot t^{1 - \alpha}  \right)^{t^{\alpha}}\cdot \left( \frac{b_2}{c_1 t} \right)^{t^\alpha} \cdot \left(1 - \frac{b_1}{c_2 t} \right)^{t - \lceil t^\alpha \rceil}\\
& < \left( \frac{B}{t^{\alpha}} \right)^{t^{\alpha}}.
\end{align*}
for some $B>0$. Notice again that this holds independently of the specific values of $b_1'$ and $b_2'$, so long as $b_1 \leq b_1' < b_2' \leq b_2$. Now we define a new event, that the tails are short (\textbf{sh}):
$$\textbf{sh}(t):=\bigwedge_{i=1}^n U_i(t+1)< t^{\alpha}.$$

After increasing $B$ to allow for the non-independence of the $n$ different $U_i$ we now see that: 
\begin{equation} \label{eqution:shbound}
\Prob \left( \neg \textbf{sh}(t) \ \Big| \Big| \ \textbf{Bo}(t,b_1(\eta),b_2(\eta)) \right) < n \cdot \left( \frac{B}{t^{\alpha}} \right)^{t^{\alpha}}.
\end{equation}

Putting these events together, define
$$\textbf{Sh}(t_2):=\forall t \in [t_2,t_2+M) \ \ \textbf{sh}(t).$$

To obtain a similar bound for $\Prob \left(\neg \textbf{Sh}(t_2) || \textbf{Bo} \left( t_2,b_1 \! \left( \tfrac{\eta}{2} \right) ,b_2 \! \left( \tfrac{\eta}{2} \right) \right) \right)$ we first show that for $j \leq M$
\begin{align} \label{equation:Shimpl}
\bigg( \textbf{Bo} \left( t_2,b_1 \! \left( \tfrac{\eta}{2} \right) ,b_2 \! \left( \tfrac{\eta}{2} \right) \right) \ & \& \ \bigwedge_{i=0}^{j-1} \textbf{sh}(t_2+j-1) \bigg) \\  \nonumber & \Rightarrow \textbf{Bo}(t_2 + j ,b_1(\eta),b_2(\eta)).
\end{align}

Suppose that $\textbf{Bo} \left( t_2,b_1 \! \left( \tfrac{\eta}{2} \right),b_2 \! \left( \tfrac{\eta}{2} \right) \right)$ holds. We address the lower bound first, for which we do not require the hypothesis on $\textbf{sh}$. Instead, for all $j \leq M$, clearly $$d(t_2 +j) \geq d(t_2) \geq b_1 \! \left( \tfrac{\eta}{2} \right) \cdot t_2 = \kappa_1 \cdot \left( y_1 - \frac{\eta}{2n} \right) \cdot t_2.$$ 
If additionally $t_2 \geq  \frac{2nM y_1}{\eta}$, then the final term above exceeds $$\kappa_1 \cdot \left( y_1 - \frac{\eta}{n} \right) \cdot (t_2+M) \geq b_1(\eta) \cdot (t_2+j).$$

Now we obtain the corresponding upper bound. By our assumption on $\textbf{sh}$,
\begin{align*}
d(t_2+j) & \leq d(t_2) + M \cdot (t_2 + M)^{\alpha}\\
& \leq  \kappa_2 \cdot \left( y_2 + \frac{\eta}{2n}\right) \cdot t_2 + M \cdot (t_2 + M)^{\alpha}\\
& \leq  \kappa_2 \cdot \left( y_2 + \frac{\eta}{n}\right) \cdot t_2
\end{align*}
if $t_2 \geq \max \left\{ M, \left( \frac{4Mn}{\kappa_2 \eta} \right)^{\frac{1}{1 - \alpha}} \right\}$, which completes the proof of Implication (\ref{equation:Shimpl}).

Implication (\ref{equation:Shimpl}) allows us to take the $M$-fold sum of (\ref{eqution:shbound}), finding  
$$\Prob \left(\neg \textbf{Sh}(t_2) \ \Big| \Big| \ \textbf{Bo} \left( t_2,b_1 \! \left( \tfrac{\eta}{2} \right) ,b_2 \! \left( \tfrac{\eta}{2} \right) \right) \right) < \sum_{t=t_2}^{t_2+M} n \left( \frac{B}{t^{\alpha}} \right)^{t^{\alpha}} \to 0$$ as $t_2 \to \infty$. Thus for any $\theta>0$ for all large enough $t_2$ we have 
\begin{equation} \label{equation:thetabound}
\Prob\left( \neg \textbf{Sh}(t_2) \ \Big| \Big| \ \textbf{Bo} \left( t_2,b_1 \! \left( \tfrac{\eta}{2} \right) ,b_2 \! \left( \tfrac{\eta}{2} \right) \right) \right) < \theta.
\end{equation}

Finally, we may complete the argument, setting $\delta = \frac{\eps}{8}$ and $\theta = \zeta = \frac{\eps}{4}$ and $\eta = \frac{ \eps}{2}$  and $t_1:=t_2+M$. For large enough $t$, we may update bound (\ref{equation:Ubound}) to get 

\begin{equation*} 
\Prob\left( \left( \neg \textbf{U}(t_2 +1)=\textbf{m} \right) \ \& \  \textbf{sh}(t_2)  \ \Big| \Big| \ \textbf{Bo} \left( t_2,b_1 \! \left( \tfrac{\eta}{2} \right) , b_2 \! \left( \tfrac{\eta}{2} \right) \right) \right)< 1 - C_4.\end{equation*}
Similarly, 
\begin{align*} 
\Prob\Big( \big( \neg \textbf{U}(t_2+j+1) =\textbf{m} \big) &\  \&  \ \textbf{sh}(t_2+j) \\  \  & \Big| \Big| \ \bigwedge_{i=0}^{j-1} \textbf{sh}(t_2+i) \ \& \  \textbf{Bo} \left( t_2,b_1 \! \left( \tfrac{\eta}{2} \right) , b_2 \! \left( \tfrac{\eta}{2} \right) \right) \Big) \\ & \ \ \ \ \ \ < 1 - C_4.\end{align*}
As observed earlier, these bounds hold independently of the previous values of $\textbf{U}$, meaning that 
\begin{align*} 
\Prob\Big( \big( & \neg \textbf{U}(t_2+j+1) =\textbf{m} \big) \  \&  \ \textbf{sh}(t_2+j) \\  \  & \Big| \Big| \ \bigwedge_{i=0}^{j} \neg \textbf{U}(t_2+i) \ \& \ \bigwedge_{i=0}^{j-1} \textbf{sh}(t_2+i) \ \& \  \textbf{Bo} \left( t_2,b_1 \! \left( \tfrac{\eta}{2} \right) , b_2 \! \left( \tfrac{\eta}{2} \right) \right) \Big) \\ & \hspace{5cm} < 1 - C_4.\end{align*}

Taking the product of these bounds, and denoting the failure of our desired result by $\textbf{Fa}(t_2):= \forall t \in [t_2, t_2+M) \  \left( \textbf{U}(t+1) \neq \textbf{m} \right)$, we see that 
$$\Prob \Big( \textbf{Fa}(t_2) \ \& \ \textbf{Sh}(t_2)  \ \Big| \Big| \ \textbf{Bo}(t_2,b_1 \! \left( \tfrac{\eta}{2} \right) , b_2 \! \left( \tfrac{\eta}{2} \right)  \Big) < \zeta$$
and so by bounds (\ref{equation:b2bound}) and (\ref{equation:thetabound})
\begin{align*}
\Prob \Big( \textbf{Fa}(t_2) \Big) &  <  \zeta + \theta + \tfrac{\eta}{2} + 2\delta = \eps.
\end{align*}
\end{proof}

We may now complete the proof of our main result, which we first restate for the reader's convenience:

\mainmulti*

\begin{proof}
First notice that there are countably many witness requests. Thus we may organise them into a list $\left( W_j : j \geq 1 \right)$.

Let $\eps >0$. Again everything that occurs is conditioned upon $G_0(t_0)$. We shall show that the probability of all witness requests eventually being satisfied exceeds $1-\eps$. Suppose inductively that we have found time $t_j$ so that so that $\Prob (\bigwedge_{i=1}^j W_i[t_j]) > 1- \left( 1 - \frac{1}{2^j} \right) \eps$.

Let $\mathcal{G} = \mathcal{G}_j$ be the set of all states $G=G(t_j)$ of the graph at time $t_j$ consistent with $G_0(t_0)$ and with $\bigwedge_{i=1}^j W_i[t_j]$. Notice that $\mathcal{G}$ is a finite set, that ${\bf P} \left( G(t_j) \  \big| \big| \ G_0(t_0) \right)>0$ for each $G \in \mathcal{G}$, and by assumption that $\sum_{G \in \mathcal{G}} {\bf P} \left( G(t_j) \  \big| \big| \ G(t_0) \right) > 1- \left( 1 - \frac{1}{2^j} \right) \eps$.

Consider now $W_{j+1}$ and let $\eps' < \frac{1}{2^{j+1}} \eps$. Now given each $G^{(k)} \in \mathcal{G}$, by Proposition \ref{prop:mainstep1} there exist $t^{(k)} \geq t_j$ such that $${\bf P} \left( W_{j+1} \left[ t^{(k)} \right] \ \big| \big| \ G^{(k)}(t_j) \right) >1 - \eps'.$$

Let $t_{j+1} := \max \{t^{(k)} \ | \ G^{(k)} \in \mathcal{G}\}$. Then 
\begin{align*}
& \Prob\left(\bigwedge_{i=1}^{j+1} W_i[t_{j+1}] \ \big| \big| \ G_0(t_0) \right) \\ 
\geq & \ \sum_{k} \Prob\left(\bigwedge_{i=1}^{j+1} W_i[t_{j+1}] \ \big| \big| \  G^{(k)}(t_j) \right) \cdot \Prob\left(G^{(k)}(t_j)\ \big| \big| \  G_0(t_0) \right)\\
= & \ \sum_{k} \Prob \left(W_{j+1}[t_{j+1}] \ \big| \big| \  G^{(k)}(t_j) \right) \cdot \Prob\left(G^{(k)}(t_j) \ \big| \big| \  G_0(t_0) \right)\\
> & \  \sum_{k} (1 - \eps') \cdot \Prob\left(G^{(k)}(t_j) \ \big| \big| \  G_0(t_0) \right)\\
> & \ (1 - \eps') \left(1- \left( 1 - \frac{1}{2^j} \right) \eps \right)
> 1- \left( 1 - \frac{1}{2^{j+1}} \right)\eps. \end{align*} \end{proof}

\section{Future Work} \label{sec:furtherwork}
We close this paper with a short discussion of possible future directions of study. As noted earlier, one goal is to translate the current work into the domain of graphs (rather than multigraphs) by proving Conjecture \ref{conj:mainconj}, which we restate:

\mainconj*

A second avenue to investigate is the limit of an $\textrm{MPA}_f$ process when $f$ is strictly between the constant case (considered by Kleinberg and Kleinberg in \cite{KK}) and the linear growth rate analysed here. A natural starting point would be the case $f(t) = \Theta \left(\sqrt{t} \right)$. One might hope somewhere within this regime to identify a connection to (a multigraph analogue of) the theory of Shelah-Spencer sparse random graphs as elucidated in \cite{JS}. The author recently established a connection between Shelah-Spencer graphs and the limits of finitary random processes in \cite{Elw2}, albeit in a context rather simpler than preferential attachment.\\

Thirdly, a central role in the contemporary study of graph limits is played by the theory of graphons, as developed, for instance, by Lov\'asz in \cite{LL}. Thus it is natural to seek to connect the current work to that body of knowledge. Although graphons as originally conceived do not allow for multi-edges, in \cite{KLS} a theory of convergence of sequences of multigraphs is developed within the broader setting of \emph{decorated graphs} and \emph{Banach space valued graphons}, so this is an initial point of contact to consider. (I am grateful to the anonymous reviewer for bringing this to my attention.)

\subsection*{Aknowledgements}
I thank this paper's anonymous reviewers for their thoughtful and helpful remarks.

\end{document}